\documentclass[11pt]{article}

\usepackage{fullpage}
\usepackage[mathscr]{eucal}
\usepackage{mathrsfs}
\usepackage{amsfonts}
\usepackage{amsmath}
\usepackage{amsthm}
\usepackage{amssymb}
\usepackage{latexsym}
\usepackage[all]{xy}
\usepackage{pstricks,pst-node}
\usepackage[mathscr]{eucal}

\theoremstyle{plain}
\newtheorem{thm}[equation]{Theorem}
\newtheorem{corr}[equation]{Corollary}
\newtheorem{prop}[equation]{Proposition}
\newtheorem{lemma}[equation]{Lemma}
\newtheorem{claim}[equation]{Claim}

\theoremstyle{definition}

\newtheorem{conj}[equation]{Conjecture}

\theoremstyle{remark}

\newtheorem{rem}[equation]{Remark}

\numberwithin{equation}{subsection}

\newcommand{\hdot}{{\:\raisebox{2pt}{\text{\circle*{1.5}}}}}
\newcommand{\idot}{{\:\raisebox{2pt}{\text{\circle*{1.5}}}}}
\newcommand{\hd}{{\:\raisebox{2pt}{\text{\circle*{1.5}}}}}

\DeclareMathOperator{\Tor}{{\mathrm{Tor}}}
\DeclareMathOperator{\Ext}{{\mathrm{Ext}}}

\DeclareMathOperator{\coh}{\mathrm{Coh}}

\DeclareMathOperator{\End}{{\mathrm{End}}}

\DeclareMathOperator{\Hom}{{\mathrm{Hom}}}

\DeclareMathOperator{\ttor}{{\scr{T}\!\textit{or}}}

\DeclareMathOperator{\hoch}{{\mathrm{Hoch}}}
\DeclareMathOperator{\Spec}{{\mathrm{Spec}}}

\newcommand{\beq}{\begin{equation}\label}
\newcommand{\eeq}{\end{equation}}

\newcommand{\iso}{{\;\stackrel{_\sim}{\to}\;}}
\newcommand{\cd}{\!\cdot\!}

\DeclareMathOperator{\alt}{\mathrm{Alt}}
\def\ccirc{{{}_{\,{}^{^\circ}}}}
\newcommand{\bo}{\mbox{$\bigotimes$}}
\renewcommand{\o}{\otimes }

\newcommand{\La}{{\mathsf{\Lambda}}}
\newcommand{\Id}{{\operatorname{Id}}}
\newcommand{\oo}{{\mathcal O}}
\newcommand{\eps}{{\varepsilon}}

\newcommand{\scr}[1]{\mathscr{#1}}

\def\ccirc{{{}_{^{\,^\circ}}}}

\newcommand{\be}{\beta }

\newcommand{\bra}{\{-,-\}}
\newcommand{\de}{{\delta}}
\newcommand{\al}{{\alpha}}

\newcommand{\cat}{{{\scr C}at_{\!}}}
\renewcommand{\aa}{{\mathcal A}}
\newcommand{\bb}{{\mathcal B}}
\newcommand{\cc}{{\mathcal C}}

\newcommand{\vi}{${\enspace\sf {(i)}}\;$}
\newcommand{\vii}{${\;\sf {(ii)}}\;$}
\newcommand{\viii}{${\sf {(iii)}}\;$}

\newcommand{\sset}{\subset}

\newcommand{\into}{\,\,\hookrightarrow\,\,}
\newcommand{\too}{\,\,\longrightarrow\,\,}
\newcommand{\mto}{\mapsto}
\newcommand{\onto}{\,\,\twoheadrightarrow\,\,}

\newcommand{\C}{{\mathbb{C}}}
\newcommand{\Z}{{\mathbb{Z}}}

\renewcommand{\k}{{\mathbb C}}
\newcommand{\kep}{\mathbb{C}_{\varepsilon}}
\newcommand{\ep}{\varepsilon}

\title{Gerstenhaber-Batalin-Vilkoviski structures on coisotropic intersections}
\author{Vladimir Baranovsky and Victor Ginzburg}

\begin{document}
\maketitle

\abstract{Let $Y,Z$ be a pair of smooth
coisotropic subvarieties in a smooth algebraic Poisson
variety $X$. We show that any data of first order deformation
of the structure sheaf $\oo_X$ to a sheaf of noncommutative algebras and
of the sheaves $\oo_Y$ and $\oo_Z$ to sheaves of right
and left modules over the deformed algebra, respectively, 
gives rise to a Batalin-Vilkoviski algebra structure
on the Tor-sheaf ${\scr{T}\!}or^{\oo_X}_\idot(\oo_Y,\oo_Z)$.
The induced Gerstenhaber bracket on the Tor-sheaf turns out to be
canonically defined; it is independent of the choices of deformations
involved. There are similar results for Ext-sheaves as well.

Our construction is motivated by, and is closely related to,
a result of Behrend-Fantechi \cite{BF}, who studied intersections
of Lagrangian submanifolds in a symplectic manifold.}

\section{Introduction}
\subsection{Main result}
Let $\mathbb{C}$ be a field of characteristic zero. We let
$\kep:= \mathbb{C}[\eps]/(\eps^2)$ denote the ring of dual numbers
and let all unlabeled tensor products stand for $\o_{\mathbb{C}}$.
Given an algebraic variety $X$,
we write $\oo_X$ for the structure sheaf, resp. $T_X$
for the tangent sheaf on $X$.

Fix  a smooth algebraic  variety $X$, over $\mathbb{C}$,
and  $P \in H^0(X, \La^2 T_X)$, a  Poisson bivector.
Thus,  there is   a Poisson $\bra: \mathcal{O}_X\times\mathcal{O}_X
\to\mathcal{O}_X$ given by the
formula $\{f,g\}=\langle P, df\wedge dg\rangle.$

Let $\mathcal{A}$ be a sheaf of (not necessarily commutative)
$\kep$-algebras
equipped with an algebra isomorphism 
 $\mathcal{A}/\ep \mathcal{A} \iso\mathcal{O}_X$
so that $\mathcal{A}$
gives a  flat deformation of 
the structure sheaf $\mathcal{O}_X$.
We require, in addition, that the  Poisson
bracket  induced by the commutator
in $\mathcal{A}$ be equal to the bracket $\{-,-\}$.
A particular example of such a deformation is the sheaf
$\mathcal{A}:= \kep\o \oo_X=\oo_X\oplus\eps\oo_X,$
equipped with multiplication given by the 
 well-known formula
$f\times g\mto f*g=fg+\frac{\eps}{2}\{f,g\}$,
for any $f,g\in \oo_X$.

Let $Z\sset X$ be a  smooth subvariety.
In this paper, we are interested in flat deformations of 
the sheaf $\oo_Z$, viewed as 
an $\oo_X$-module supported on $Z$, to either left or right $\aa$-module
$\cc$  set theoretically supported on $Z$.  
Associated with such a deformation
$\cc$ to a left  $\aa$-module, there is a \textit{transposed} 
deformation $\cc^t$, which gives a right $\mathcal{A}$-module,
see section \ref{transposed} for the definition of
 $\cc^t$.

Next,
let $Y,Z\sset X$ be a pair of smooth subvarieties.
Then, $Y\cap Z$, a {\em scheme theoretic} intersection,
is a closed subscheme of $X$ with structure sheaf
$\oo_{Y\cap Z}:=\mathcal{O}_Y\bo_{\mathcal{O}_X}
\mathcal{O}_Z.$   More generally, we  have
${\scr{T}\!}or^{\mathcal{O}_X}_\idot(\mathcal{O}_Y, 
\mathcal{O}_Z)$, a coherent sheaf of supercommutative graded
$\oo_{Y\cap Z}$-algebras. 
Similarly one has a  sheaf
$\scr{E}^{\!}xt_{\mathcal{O}_X}^\hd (\mathcal{O}_Y, \mathcal{O}_Z)$
that comes equipped with the natural structure
of a graded ${\scr{T}\!}or^{\mathcal{O}_X}_\idot(\mathcal{O}_Y, 
\mathcal{O}_Z)$-module (the module structure is recalled 
in Section 3.2).

Recall that a graded commutative algebra $D=\bigoplus_{k\geq 0} D_k$
equipped with an operator $\de: D_\idot\to D_{\idot-1}$
is called a Batalin-Vilkovisky (BV) algebra if $\de$ is a differential
operator of order $\leq 2$ (with respect to
multiplication in $D$) and one has $\de^2=0$.
In this case, the formula (for $x,y$ homogeneous):
\beq{bracket}
[x, y] := \delta(x\cd y) - \delta(x)\cd  y - (-1)^{\deg x}x\cd  \delta(y),
\eeq
provides $D_\idot$  with a
structure of  Gerstenhaber algebra (i.e., odd Poisson algebra).
See e.g. \cite{BF} for more details on these definitions.

Similarly, given a graded $D$-module $M=\bigoplus_{k\geq 0} M_k$,
 a BV-module structure on $M$
is the data of a linear operator $\de': M_\idot\to M_{\idot-1}$
such that $(\de')^2=0$ and such that $\de'$ has  order $\leq 2$
in the sense that
for any homogeneous $x, y \in D$ and $m \in M$, the following equation
holds
\begin{multline*}
\delta(xy) m - (-1)^{\deg x}x \delta(y)m - \delta(x) y m =  
\delta'(xym) - \\ - (-1)^{\deg x} x \delta'(ym) - 
(-1)^{\deg x \deg y + \deg y} 
y \delta'(xm) + (-1)^{\deg x + \deg y}xy \delta'(m).
\end{multline*}
In such a case,  an analogue of formula \eqref{bracket} (for $\de'$
instead of $\de$) gives a pairing
$\bra:\ D\o M \to M$ that makes $M$ a Gerstenhaber module over $D$.

The main result of this paper reads

\begin{thm}\label{BV}  Let $X$ be a smooth Poisson
variety with Poisson bivector $P$,
 and let $\mathcal{A}$ be a flat $\kep$-deformation 
of $\oo_X$ such that
the commutator in $\mathcal{A}$ induces the Poisson bracket
given by $P$.
Let $Y, Z$ be  a pair of smooth 
coisotropic subvarieties in a smooth Poisson variety $X$. 
Then, we have

\vi Associated with the data of a
flat $\kep$-deformation of the
sheaf $\oo_Y$  to a right
$\mathcal{A}$-module $\bb$ and of the
sheaf  $\oo_Z$ to a left
$\mathcal{A}$-module $\cc$, there is a second order differential
operator $\de: {\scr{T}\!}or^{\mathcal{O}_X}_\idot (\mathcal{O}_Y,
\mathcal{O}_Z)\to {\scr{T}\!}or^{\mathcal{O}_X}_{\idot-1}
(\mathcal{O}_Y, \mathcal{O}_Z)$ that squares to zero
(i.e. makes ${\scr{T}\!}or^{\mathcal{O}_X}_\idot (\mathcal{O}_Y,
\mathcal{O}_Z)$  a  BV algebra) provided
the first order deformations locally admit extensions to second
order deformations.

\vii The induced  bracket \eqref{bracket} on
${\scr{T}\!}or^{\mathcal{O}_X}_\idot (\mathcal{O}_Y,
\mathcal{O}_Z)$
is independent of the choice of  deformations $\bb$ and $\cc$.

\viii Similarly, given an additional flat $\kep$-deformation
of $\oo_Z$ to a {\em right} $\aa$-module $\cc'$,
there is an associated second order differential
operator
$\delta': \
\scr{E}^{\!}xt^\hdot_{\mathcal{O}_X} (\mathcal{O}_Y, \oo_Z)\to
\scr{E}^{\!}xt^{\hdot-1}_{\mathcal{O}_X} (\mathcal{O}_Y, \oo_Z)$,
such that $(\de')^2=0$ if the first order
deformations locally admit extensions to second 
order deformations. If $\cc' = \cc^t$ 
the corresponding operator $\delta'$ provides
the sheaf
 $\scr{E}^{\!}xt^\hdot_{\mathcal{O}_X} (\mathcal{O}_Y, 
\oo_Z)$ with a structure of BV-module over
the BV algebra ${\scr{T}\!}or^{\mathcal{O}_X}_\idot (\mathcal{O}_Y,
\mathcal{O}_Z)$.
Moreover, the resulting pairing
$$\bra:\
{\scr{T}\!}or^{\mathcal{O}_X}_\idot (\mathcal{O}_Y,
\mathcal{O}_Z)\times\scr{E}^{\!}xt^\hdot_{\mathcal{O}_X} (\mathcal{O}_Y, 
\oo_Z)\to
\scr{E}^{\!}xt^\hdot_{\mathcal{O}_X} (\mathcal{O}_Y, 
\oo_Z)$$
is independent of the choice of deformations $\bb$ and $\cc$.
\end{thm}

Since flat deformations exist locally,
Theorem \ref{BV} yields the following corollary, which is  the
second
important result of the paper.

\begin{corr}\label{gerst_cor}
Let  $Y,Z\sset X$ any pair of smooth
coisotropic submanifolds of an arbitrary smooth Poisson
algebraic variety
$X$.  Then, on
 ${\scr{T}\!}or^{\mathcal{O}_X}_\idot (\mathcal{O}_Y,
\mathcal{O}_Z)$ there is a {\em canonical}
structure of  Gerstenhaber algebra.

Furthermore, the group
 $\scr{E}^{\!}xt^\hdot_{\mathcal{O}_X} (\mathcal{O}_Y, \oo_Z)$ 
has  a {\em canonical}
structure of  Gerstenhaber module
over the Gerstenhaber algebra ${\scr{T}\!}or^{\mathcal{O}_X}_\idot (\mathcal{O}_Y,
\mathcal{O}_Z)$.
\end{corr}

\noindent
Several examples of such BV and Gerstenhaber structures
are discussed in \S\ref{Exam} below.
\medskip

Our results above were, to a great extent,  inspired
by the work   of K. Behrend and B. Fantechi \cite{BF}.
Behrend and  Fantechi consider a pair
 $Y,Z,$ of  Lagrangian submanifolds
in a {\em holomorphic}
 symplectic  manifold $X$.
They show that one can equip the graded
algebra ${\scr{T}\!}or^{\oo_X}_\idot(\oo_Y,\oo_Z)$ with
a Gerstenhaber bracket, resp.
the graded sheaf $\scr{E}^{\!}xt_{\oo_X}^\hdot(\oo_Y,\oo_Z)$
with a BV type differential.
The approach in \cite{BF} is based on reducing the case of
 a general Lagrangian intersection
to the special case where $X=T^*Y$
and $Z\sset T^*Y$ is the graph of a 
holomorphic function on $Y$
 (and $Y$ is identified with the zero section of $T^*Y$).

Thus, the arguments in \cite{BF}
rely in a crucial way on a version of
Darboux theorem saying that
any holomorphic symplectic manifold is locally isomorphic
to a cotangent bundle. 
Such a result holds for holomorphic symplectic manifolds
(equipped with the
usual Hausdorff topology)
but it is totally false in the algebraic setting.
Indeed, an algebraic symplectic 2-form
need not be locally exact, even in
\`etale topology.
The corresponding argument,
kindly communicated to us by A. Beilinson,
will be given in section \ref{sasha}.

\subsection{Construction of the BV differential}\label{bvdif}
Let $\aa$ be any   flat $\kep$-deformation of the sheaf
$\mathcal{O}_X$ to a sheaf of associative $\kep$-algebras
equipped with an algebra isomorphism
$\mathcal{A}/\ep\mathcal{A} \simeq \oo_X$.
Similarly, let 
 $\mathcal{B}$ be a flat deformation of  $\mathcal{O}_Y$
to a right $\mathcal{A}$-module and $\oo_Z$ has a flat
deformation $\mathcal{C}$ to a left $\mathcal{A}$-module. 
The flatness assumptions imply that multiplication 
by $\eps$ induces an isomorphism $\oo_Y=\cc/\eps\cc\iso \eps\cc,$
and similar isomorphisms $\oo_X\iso\eps\aa$.

The short exact sequence $0 \to \ep \mathcal{C} \to \mathcal{C} \to 
\mathcal{C}/\ep \mathcal{C} \to 0$ induces a long exact sequence
$$
\ldots \to 
{\scr{T}\!}or_{i+1}^{\mathcal{A}}(\mathcal{B}, \mathcal{C}/\ep \mathcal{C})
\to {\scr{T}\!}or_{i}^{\mathcal{A}}(\mathcal{B}, \ep \mathcal{C}) \to
{\scr{T}\!}or_{i}^{\mathcal{A}}(\mathcal{B}, \mathcal{C}) \to
{\scr{T}\!}or_{i}^{\mathcal{A}}(\mathcal{B}, \mathcal{C}/\ep \mathcal{C})
\to \ldots.
$$

Locally, we can choose a projective resolution $P^\hdot$ of $\mathcal{B}$
with $\mathcal{A}$-modules, such that $P^\hdot/\ep P^\hdot$ is a
resolution of $\mathcal{O}_Y$ with projective $\mathcal{O}_X$-modules. 
Further, we have an isomorphism of functors
 $(\cdot) \otimes_{\mathcal{A}} \ep \mathcal{C} 
\simeq (\cdot) \otimes_{\mathcal{A}} \mathcal{O}_X
\otimes_{\mathcal{O}_X} \oo_Z$
and similarly for $\mathcal{C}/\ep \mathcal{C}$.
We deduce canonical isomorphisms
$$
{\scr{T}\!}or^{\mathcal{A}}_i (\mathcal{B}, \ep\mathcal{C})
\simeq {\scr{T}\!}or^{\mathcal{O}_X}_i (\mathcal{O}_Y, \mathcal{O}_Z) \simeq
{\scr{T}\!}or^{\mathcal{A}}_i (\mathcal{B}, \mathcal{C}/\ep\mathcal{C})
$$

Using these isomorphisms, the connecting morphism
in the long exact sequence above yields a map
$$
\delta: {\scr{T}\!}or^{\mathcal{O}_X}_{i+1} (\mathcal{O}_Y, \mathcal{O}_Z)
\to 
{\scr{T}\!}or^{\mathcal{O}_X}_i (\mathcal{O}_Y, \mathcal{O}_Z).
$$

Similarly, suppose that we have 
a deformation $\mathcal{C}'$ of $\mathcal{O}_Z$ to a \textit{right}
$\mathcal{A}$-module. Then there is a long exact sequence
$$
\ldots \to 
\scr{E}^{\!}xt^{i-1}_{\mathcal{A}}(\mathcal{B}, \mathcal{C}'/\ep \mathcal{C}')
\to \scr{E}^{\!}xt^{i}_{\mathcal{A}}(\mathcal{B}, \ep \mathcal{C}') \to
\scr{E}^{\!}xt^{i}_{\mathcal{A}}(\mathcal{B}, \mathcal{C}') \to
\scr{E}^{\!}xt^{i}_{\mathcal{A}}(\mathcal{B}, \mathcal{C}'/\ep \mathcal{C}')
\to \ldots
$$
In particular,  one has a morphism
$$
\delta':\
 \scr{E}^{\!}xt_{\mathcal{A}}^i(\mathcal{B}, \ep \mathcal{C}')
\simeq \scr{E}^{\!}xt_{\mathcal{O}_X}^i (\mathcal{O}_Y, \mathcal{O}_Z)
\too \scr{E}^{\!}xt_{\mathcal{O}_X}^{i+1} (\mathcal{O}_Y, \mathcal{O}_Z)
\simeq \scr{E}^{\!}xt_{\mathcal{A}}^{i+1}(\mathcal{B}, \mathcal{C}'/\ep \mathcal{C}')
$$
When both $\mathcal{C}$ and $\mathcal{C}'$ are given 
we will assume that $\mathcal{C}'=\cc^{t}$ or $\cc'$ is 
\textit{transposed to} $\mathcal{C}$. 
See Section 3.2 regarding the canonical product on $\Tor$ and 
its action on $\Ext$. 

\subsection*{Notation.}

Given a vector bundle (a locally free sheaf) $E$, we write $E^\vee$
for the dual vector bundle. Let $\Omega_X,$ resp. 
$T_X=\Omega_X^\vee$, denote the cotangent, resp. tangent,
sheaf on a manifold $X$. Let $N_{X/Y}$ denote
the normal sheaf for a submanifold $Y\sset X$.

We often abuse the notation and write
$\Tor_\idot^{X}(-,-)$ for $\Tor_\idot^{\oo_X}(-,-),$
and semilarly for Ext's.

\subsection{A conjecture by physicists}
Recall first that, for any (triangulated) category
$\scr C$, one can define its 
Hochschild cohomology
groups $H\!H^\hdot(\scr C)$.

According to  A. Kapustin and L. Rozansky 
one has the following

\begin{conj}{\em To each pair $Y,Z\sset X,$ of
 smooth Lagrangian submanifolds
of  a smooth  algebraic symplectic variety $X$, one can associate
a triangulated category $\cat_X(Y,Z)$, cf. \cite{KR}, such
that 
the Hochschild cohomology
of the  category $\cat_X(Y,Z)$ is given by
$$H\!H^\hd(\cat_X(Y,Z))\cong \ttor_\idot^{\oo_X}(\oo_Y,\oo_Z).
$$

Moreover,
the  standard Gerstenhaber bracket on Hochschild cohomology
goes, under the isomorphism above, to the
canonical Gerstenhaber bracket 
on $\ttor_\idot^{\oo_X}(\oo_Y,\oo_Z)$ 
provided by Corollary \ref{gerst_cor}.}
\end{conj}

If $X=T^\vee Y$ is the cotangent bundle of $Y$ 
and $Z=Y$ is the zero section, then one should have
$\cat_X(Y,Z)=D^b(\coh Y)$.
In this case, we have
$$HH^\hd(D^b(\coh Y))=H^\hd(Y, \La^\hd T_Y)=
\Tor_\idot^{T^\vee Y}(\oo_Y,\oo_Y).
$$
so that the Gerstenhaber bracket is induced by  the Schouten bracket
on $\La^\hd T_Y$.

More generally, let  $X=T^\vee Y$ and $Y$ be the zero section as above,
and let $Z=\text{Graph}(df)$ where $f\in \k[Y]$.
Then $\cat_X(Y,Z)$ should be the
category of matrix factorizations 
$(F \overset{_d}{\underset{^{d'}}{\leftrightarrows}} F',$
$d\ccirc d'=f\cdot\Id=d'\ccirc d),$ associated with the function $f$,
cf. \cite{O}.

\begin{rem}
Observe that the sheaves ${\scr{T}\!}or_\idot^{\mathcal{O}_X}
(\mathcal{O}_Y, \mathcal{O}_Z)$ are related to the  hyper-Tor groups
$\Tor_\idot^{X}(\mathcal{O}_Y, \mathcal{O}_Z)$ via the local-to-global 
spectral sequence 
$$
H^\idot(X,\, {\scr{T}\!}or_\idot^{\mathcal{O}_X}
(\mathcal{O}_Y, \mathcal{O}_Z))
\enspace\Longrightarrow \enspace
\Tor_\idot^{X}(\mathcal{O}_Y, \mathcal{O}_Z).$$
At the same time, the global hyper-Tor may also be calculated by
applying $R\Gamma$ to the sheaf of DG algebras $T_\idot$ described
in Section 3 below. Thus, we expect that there exists a refined version
of our results, in which Gerstenhaber or Batalin-Vilkovisky structures
on the cohomology sheaves of $T_\idot$, are replaced by their ``strong
homotopy" versions on $T_\idot$ itself. 
In fact, the lemmas of Section 3.2 point towards such a refinement.
Similar remarks apply to 
the Ext groups (local and global), and
 the polydifferential version of the resolution $E^\hdot$ 
in Section 3. 
\end{rem}

\section{Existence of first and second order deformations.}

\subsection{Algebraic setup}\label{setup}
Following Gerstenhaber, a $\kep$-flat deformation
of an associative
 $\C$-algebra $A$ is given by a $\kep$-bilinear associative
 product structure on the vector space
 $A_\ep = A\oplus \ep A$  defined by 
$$
(a_1 \oplus 0) * 
(a_2  \oplus 0) = a_1 a_2 \oplus \ep\cdot \alpha_A(a_1, a_2), \qquad a_1, a_2 \in A 
\subset A_\ep
$$ 
with $\alpha_A: A\otimes A\to A$ a $\C$-linear map.
The  associativity of the $*$-product is equivalent to the 
equation
\begin{equation} 
\label{hoch-A}
\alpha_A (a_1 a_2, a_3) - \alpha_A (a_1, a_2 a_3) + \alpha_A (a_1, a_2) a_3 - 
a_1 \alpha_A (a_2, a_3) = 0.
\end{equation} 

Fix a $\kep$-flat deformation of $A$ as above.
Given  a right $A$-module $B$, one may consider
$\kep$-flat extensions of the $A$-module structure on $B$
to a  right  $A_\ep$-module on  $B_\ep = B\oplus\ep B$.
Explicitly, such an $A_\ep$-module structure  on  $B_\ep$
is determined by
a bilinear map $\alpha_B: B\otimes A \to B$.
The corresponding right  $A_\ep$-action  is given by
the formula
$$
(b \oplus 0)* 
(a \oplus 0) = ba \oplus \ep\cdot \alpha_B(b, a), \quad a\in A, b\in B.  
$$
The map $\alpha_B: B\otimes A \to B$ must satisfy the 
associativity equation
\begin{equation}
\label{hoch-B}
\alpha_B (b a_2, a_3) - \alpha_B (b, a_2 a_3) + \alpha_B (b, a_2) a_3 - 
b \alpha_A (a_2, a_3) = 0.
\end{equation}

Further, 
any pair of $\kep$-linear
automorphisms of the form
$$
A_\ep \to A_\ep,\enspace
a \mapsto a \oplus \ep \beta_A(a) \quad
\text{and}\quad B_\ep \to B_\ep,\enspace
b \mapsto b \oplus \ep \beta_B(b),
$$
where  $\beta_A: A \to A$ and $\beta_B: B \to B$ are $\C$-linear
maps, induces equivalent
deformations corresponding to cochains
\begin{equation}
\label{auto-A}
(a_1, a_2) \mapsto \alpha_A(a_1, a_2) + \beta_A(a_1 a_2) - a_1 
\beta_A(a_2) - \beta_A(a_1) a_2
\end{equation}
\begin{equation}
\label{auto-B}
(b, a_2) \mapsto \alpha_B(b, a_2) + \beta_B(b a_2) - b 
\beta_A(a_2) - \beta_B(b) a_2
\end{equation}

A deformation as above
 extends to $\C[\varepsilon]/(\varepsilon^3)$ if and only 
if one has
\begin{equation}
\label{integrab-A}
\alpha_A(\alpha_A(a_1, a_2), a_3) -
\alpha_A(a_1, \alpha_A(a_2, a_3)) = d \gamma_A (a_1, a_2, a_3)
\end{equation}
\begin{equation}
\label{integrab-B}
\alpha_B(\alpha_B(b, a_1), a_2) -
\alpha_B(b, \alpha_A(a_1, a_2)) = d \gamma_B (b, a_1, a_2),
\end{equation}
where $\gamma_A: A \otimes A \to A$, $\gamma_B: B\otimes A \to B$
are some linear maps
and $d\gamma_A$, $d\gamma_B$ are defined similarly to the LHS of 
\eqref{hoch-A} and \eqref{hoch-B}, respectively. 

\subsection{Deformation complex}\label{def_comp_sec}

The identities of the previous subsection can be reformulated as follows.
The $A$-module structure on $B$ defines a homomorphism $g: A\to \End_\C B$
of algebras over $\C$, and deforming the algebra/module structure 
amounts to deforming $g$ to an algebra homomorphism $A_\ep 
\to \End_{\kep}(B_\ep)$. Observe that $\End_{\kep} (B_\ep)$ is the 
trivial deformation of $\End_\C B$. Thus, adjusting the definitions of
\cite{MR704600}, \cite{FMY} (i.e. 
removing the term responsible for the deformation of the algebra
$\End_\C B$) we introduce a deformation complex $C^\hdot_{A,B}$
with terms
$$
C^n_{A,B} = C^n(A, A) \oplus C^{n-1} (A, \End_\C B) = \Hom_\C(A^{\otimes n}, A)
\oplus \Hom_\C(A^{\otimes (n-1)}, \End_\C B),
$$
where $C^n(A, X)$ denotes the standard complex of Hochschild cochains of 
an $A$-bimodule $X$. The differential in the complex $C^\hdot_{A,B}$ is given by
\beq{d_g}
d_{A,B}(\alpha_A \oplus \alpha_B) = d_{_{\operatorname{Hoch}}} \alpha_A \oplus (g \alpha_A 
-  d_{_{\operatorname{Hoch}}} \alpha_B),
\eeq
where $d_{_{\operatorname{Hoch}}}$ stands for
 the standard Hochschild differential, cf. \textit{loc. cit.}
We put $H^n_{A,B} := H^n(C^\hdot_{A,B})$. 

Equations
\eqref{hoch-A}, \eqref{hoch-B} say that 
$\alpha = \alpha_A \oplus \alpha_B$ is a cocycle in $C^2_{A,B}$.
Equations
\eqref{auto-A} and \eqref{auto-B} say that the equivalence class 
of the deformation depends only on the image of $\alpha$ in $H^2_{A,B}$. 

To reinterpret integrability conditions recall that by 
\textit{loc. cit.} $C^{\hdot - 1}_{A,B}$ has a structure
of  DG Lie algebra such that the term $C^{\hdot -1}(A, A)$ 
with its Gerstenhaber bracket, is a quotient DG Lie algebra
of $C^{\hdot-1}_{A,B}$. Explicitly, 
up to a choice of signs, for $\alpha_A \oplus \alpha_B \in C^n_{A,B}$, 
$\alpha'_A \oplus \alpha'_B \in C^m_{A,B}$ one has
\begin{multline*}
[\alpha_A \oplus \alpha_B, \alpha'_A \oplus \alpha'_B] = 
\big(\alpha_A \circ \alpha_A' - (-1)^{(n-1)(m-1)} \alpha_A' \circ \alpha_A\big)
\oplus 
\\ \oplus \big(\alpha_B \circ \alpha_A' + \alpha_B \cup \alpha_B' - (-1)^{(n-1)(m-1)} (\alpha_B' \circ \alpha_A + \alpha_B' \cup \alpha_B)\big) 
\end{multline*}
where
$$
\alpha_B \circ \alpha_A' := \sum_{s=1}^n (-1)^{(s-1)(m-1)} 
\alpha_B(1_A^{\otimes (s-1)} 
\otimes \alpha_A' \otimes 1^{\otimes (n-s)}_A)
$$
and similarly for the other terms. The cup product 
$\alpha_B\cup \alpha_A': B \otimes A^{\otimes^{m+n-1}} \to B$
is the composition of $\alpha_B \otimes \alpha_A': B \otimes A^{\otimes^{m+n-1}}
\to B \otimes A$ and the action map $B\otimes A \to B$.  

\bigskip
\noindent
Then, equations \eqref{integrab-A}, \eqref{integrab-B}
say that $\frac{1}{2}[\alpha, \alpha] = d_{A,B} (\gamma)$, i.e. that 
$[\alpha, \alpha]$ represents the zero class in $H^3_{A,B}$. 

\bigskip
\noindent
Observe that $C^{\hdot -1}(A, \End_\C B)$  is 
a subcomplex of $C^\hdot_{A,B}$,
and 
$C^\hdot(A, A)$ is a quotient complex of $C^\hdot_{A,B}.$
 The corresponding
long exact sequence of cohomology reads 
\beq{long1}\ldots\to \Ext^{n - 1}_A (B, B)
\to H^n_{A,B}\to H^n(A,A)\stackrel{g}\to\Ext^n_A (B, B)\to\ldots
\eeq

We see that  an $n$-cocycle $\alpha_A\in C^2(A,A)$ may be
lifted to a class in $H^2_{A,B}$ if and only if
the map $g\ccirc \alpha_A: A\times A\to \End_\C B$,
that represents the image of the
class of $\alpha_A$ under the connecting homomorphism,
gives the zero class in $\Ext^2_A(B,B)$.

Similarly, given another algebra homomorphism $h: A\to ( \End_\C C)^{op}$,
one can introduce a bigger deformation complex with terms
\beq{c_gh}
C_{A,B,C}^n = 
 C^n(A, A) \oplus C^{n-1} (A, \End_\C B) \oplus
C^{n-1} (A, \End_\C C).
\eeq
The differential in the complex $C^\hdot_{A,B,C}$ is given by
$$
d_{A,B,C}(\alpha_A \oplus \alpha_B
\oplus \alpha_C) = d_{_{\operatorname{Hoch}}} \alpha_A \oplus (g \alpha_A 
-  d_{_{\operatorname{Hoch}}} \alpha_B) \oplus (h\alpha_A  - d_{_{\operatorname{Hoch}}} \alpha_C).
$$
For the corresponding cohomology groups $H^\hd_{A,B,C}$
there is a long exact sequence
\begin{multline*}
\ldots\to \Ext^{n - 1}_A (B, B)\oplus
\Ext^{n - 1}_A (C,C)
\too H^n_{A,B,C}\too\\
\too  H^n(A,A)\stackrel{g\oplus h}\too\Ext^n_A (B, B)\oplus
\Ext^{n}_A (C,C)\to\ldots
\end{multline*}

\subsection{Local deformations}

Let now $X$ be a smooth affine  variety
and $Y\sset X$ a smooth closed subvariety. 
Write $A:=\C[X]$, resp. $B:=\C[Y]$, for the
corresponding coordinate rings.

A bivector $P\in H^0(X, \La^2 T_X)$
with a vanishing Schouten bracket gives a Poisson structure on $A$.
We will say that $Y$ is \textit{coisotropic} with respect to $P$
if $P$ projects to zero in $H^0(Y, \La^2 N_{X/Y})$.

\begin{prop}
\label{exist} The map $H^\hdot_{A,B}\to H^\bullet(A,A)$,
in \eqref{long1}, is injective. Furthermore,
for any  $2$-cocycle
$\alpha_A \in C^2(A, A)$ 
the following holds:
\vskip 1pt

\vi $Y$ is a {\em coisotropic} subvariety in $X$
if and only if  there exists 
$\alpha_B: B \otimes A\to B$ that gives a first order
deformation of
 $B$, i.e., if and only if 
there exists $\alpha_B$ such that the pair
$(\alpha_A, \alpha_B)$ gives a 2-cocycle in the
complex $C^\hdot_{A,B}$.

\vii Assume that 
 $\alpha_A(a_1, a_2) = \frac{1}{2} \langle P, da_1 \wedge da_2 \rangle$ 
with
 $P\in H^0(X, \La^2 T_X)$. 
Then, in  $\mathsf{(i)}$,  one may choose $\alpha_B: B\otimes A \to B$
to be a sum of
a bidifferential operator of bidegree $(1, 1)$ and a bidifferential
operator of bidegree $(0, 2)$.

\viii Assume, in addition, that the bivector $P$ 
has a vanishing Schouten bracket with itself: $\{P, P\} = 0$.
Then, there exists a {\em symmetric} bilinear map
 $\gamma_A: A\otimes A \to A$ such that equation
\eqref{integrab-A} holds.
If, moreover, the
canonical class of $Y$ is trivial, then there exist $\alpha_B: B \otimes A
\to B$ and $\gamma_B: B \otimes A \to B$ such that  equations
\eqref{hoch-B} and 
 \eqref{integrab-B} hold, i.e., the map
$(b, a)\mapsto ba + \varepsilon \alpha_B(b, a) + \varepsilon^2 
\gamma_B(b, a)$ gives a second order deformation of $B$.
\end{prop}
\begin{proof} First,
note that $g \alpha_A =
 d_{_{\operatorname{Hoch}}}
 \alpha_B$ 
in $C^2(A, \End_\C B)$ means that \eqref{hoch-B} holds by  
definition of the  maps involved. This
proves (i).

\bigskip
\noindent
To prove $(ii)$ let $I\sset A$ denote the defining ideal
of the subvariety $Y$.

First, we are going to construct
a map $\alpha^0_B:\ B \otimes I \to B,$
the restriction of the cocycle $\alpha_B:\ B \o A\to B$, that we are
looking for, to $B\o I$.
Observe that the cocycle equation \eqref{hoch-B} implies that,
 the map $\alpha^0_B$ should satisfy the following two
constraints:
$$
\alpha^0_B(ba, x) - \alpha^0_B(b, ax) - b \alpha_A(a, x) = 0;\quad
- \alpha^0_B(b, xa) + \alpha^0_B(b, x) a - b \alpha_A(x, a) = 0,
$$
for any 
$b \in B,\, a \in A,\,x \in
I.$

We will define   $\alpha^0_B$ to be a map of the following form:
$$
\alpha^0_B(b, x) = \rho (db, x) + b \psi(x);
$$ 
where $\rho: \Omega^1_B \otimes (I/I^2) \to B$ is a $B$-bilinear map
and $\psi: I/I^2 \to B$ is a first order algebraic 
differential operator.
In terms of $\rho$ and $\psi$, the two constraints above translate into
the following pair of equations, for any
$a \in A,\,x \in
I,$
\beq{rp}
\psi(ax) - a \psi(x) = 1_B \alpha_X(a, x); \qquad \rho(d(a|_Y), x) = 2 
\cdot 1_B \alpha_X(a, x).
\eeq

We remark that the second equation in \eqref{rp} determines
$\rho$ uniquely, since every element of $B$ is an image of some $a \in
A$. Observe further that, for  $a,x \in I$,
we have $\alpha_A(a, x) \in I$ since $Y$ is a coisotropic subvariety.
Hence, in this case $1_B \alpha_A(a, x) = 0$.
We see that we may  (and will) use the second equation in \eqref{rp}
as a definition of $\rho$; the resulting map
$\rho$ is well-defined.

Observe next that  the first equation in \eqref{rp}
is a condition on the map $\sigma_\psi:\
I/I^2 \otimes \Omega^1_B \to B$,  the principal symbol
 of the first order differential operator $\psi$.
Specifically, the equation
says that $\sigma_\psi(a,x)=\frac{1}{2}\langle P|_Y,\, da\wedge dx\rangle$
for any $a\in I, x\in B$.
Again, we may (and will) use the latter equation as
the definition of $\sigma_\psi$. The resulting symbol
is well-defined since $P$ sends $I \otimes I$ to zero in $B$.

Recall next that,
for any  $\sigma_\psi$, one may find a differential
operator $\psi$ that has  $\sigma_\psi$ as its principal symbol.
Indeed, let ${\scr D}^{\leq 1}(N^\vee_{X/Y}, \mathcal{O}_Y)$
denote the space of first order algebraic differential operators
$N^\vee_{X/Y}\to\mathcal{O}_Y$.
The variety $Y$ being smooth and affine, one has 
 a short exact sequence, cf. \cite{MR0238860},
$$
0 \to Hom_{\mathcal{O}_Y}(N^\vee_{X/Y}, \mathcal{O}_Y) 
\to {\scr D}^{\leq 1}(N^\vee_{X/Y}, \mathcal{O}_Y)
\to Hom_{\mathcal{O}_Y}(N^\vee_{X/Y}\otimes_{\mathcal{O}_Y}
\Omega^1_Y, \mathcal{O}_Y) \to 0, 
$$
where the last arrow is the principal symbol map which is, therefore,
surjective. 

This completes the construction of the map $\alpha^0_B: B \otimes I \to
 B$.

\bigskip
\noindent
It remains to extend  $\alpha^0_B$ to construct
a cocycle $\alpha_B: B \otimes A \to
B$. To that end, 
note that since $Y$ is smooth and affine we can 
 choose a splitting of the short exact sequence
$$
0 \to N^\vee_{X/Y} \to \Omega^1_{X}|_Y \to \Omega^1_{Y} \to 0.
$$
Such a splitting yields a $B$-linear map 
$p: \Omega^1_A \otimes_A B \to I/I^2$.
Similarly, a splitting of the projection 
$T^*_X|_Y \onto T^*_Y$ yields  a $B$-linear map 
$q: \Omega^1_B \to \Omega^1_A \otimes_A B$.

Using the splittings, we define
$$
\alpha_B(b, a) = \rho(db, p(da)) - \frac{1}{2}P(q(da|_Y), q(db)) + b 
\psi(p(da)).
$$
It is clear that the resulting map $\alpha_B$
is an extension of $\alpha^0_B$.
Furthermore, an explicit calculation using identities
\eqref{rp}  shows that  the map $\alpha_B$ so defined
satisfies the requirements  of part $(ii)$ of the proposition.

\bigskip
\noindent
To prove part $(iii)$ we need to recall an explicit
version of the Hochschild-Kostant-Rosenberg isomorphism
for Hochschild cohomology.

The Hochschild complex that we are interested in is
the complex with terms $C^k(A, \End_C B)=Hom_\C(B \otimes A^{\otimes
k},B)$, equipped with the Hochschild
differential $d_{\hoch}.$ The Hochschild-Kostant-Rosenberg
theorem says that,
for $Y=\Spec B$  smooth, one has an isomorphism
$$
\alt:\ H^k(C^\hdot(A, \End_C B))\iso H^0(Y, \La^k N_{X/Y}).
$$
The isomorphism is constructed as follows.
Given, 
$\gamma\in Hom_\C(B \otimes A^{\otimes
k},B)$, one obtains, by restriction to the ideal of $Y$,
a polylinear map
$B \otimes I^{\otimes
k}\to B$. Let $\alt(\gamma)$ be the anti-symmetrization of
this map with respect to the last $k$ arguments.
One shows, that if $\gamma$ is a Hochschild
{\em cocycle}, i.e. $d_{\hoch}\gamma=0$, then
$\alt(\gamma)(b,x_1,\ldots,x_k)=0$ whenever
$x_i\in I^2$ for at least one  $i=1,\ldots, k$.
It follows that the map
$\alt(\gamma)$ descends to a   map
$\alt(\gamma):\ B\o \La^k_\C (I/I^2)\to B$.
Furthermore, the equation $d_{\hoch}\gamma=0$ insures that
the resulting map is 
$B$-polylinear, cf. \cite[Proposition 1.3.12]{Lo}
for a similar result in the case of Hochschild homology. 
We conclude that the map $\alt(\gamma)$ descends
 to a well-defined $B$-linear map $\alt(\gamma):\ \La_B^k(I/I^2)\to B$.
Giving such a map is the same as giving a section
of $\La^k N_{X/Y}$, and we are done.

We can now resume the proof of part (iii).
First of all, we
note that existence of \textit{some}  $\gamma_A$
is well-known, cf. e.g.  \cite{MR2062626}. By skew symmetry of 
$\alpha_A=\frac{1}{2}P$
 is follows immediately that $\frac{1}{2}(\gamma_A(a, b) +
\gamma_A(b, a))$ solves the same equation \eqref{integrab-A}.
Hence, from now on,
we assume that the bilinear map $\gamma_A$ is symmetric.

Assume now that the canonical bundle on $Y$
is trivial and choose a trivialization, that is, a
nowhere vanishing top degree differential form $\omega$ on $Y$. 
Let  $Lie_\partial(\omega)$ denote the 
Lie derivative of  $\omega$ with respect to
a vector field $\partial$ on $Y$. 
The assignment $\partial\mto Lie_\partial(\omega) \cdot \omega^{-1}$
gives a first order differential operator $T_Y \to B$.
We compose this  differential operator with the $B$-linear map $I/I^2 \to T_Y$ 
given by the restriction of the bivector $P$ to $Y$.
This way, we obtain 
a first order differential 
 operator $\psi: I/I^2 \to B$ that satisfies the identity 
$\psi(ax) - a \psi(x) = 1_B \alpha_X(a, x)$ (the first equation in
\eqref{rp}).

We use the 
above operator $\psi$ to construct a cocycle $\alpha_B: B \otimes A \to B$
following the procedure explained in the proof of part (ii).
The
resulting operator $\alpha_B: B \otimes A \to B$ satisfies
\begin{equation}
\label{pseudoLie}
\alpha_B(\alpha_B(b, x_1), x_2) - \alpha_B(\alpha_B(b, x_2), x_1)
- 2 \alpha_B(b,  \alpha_A(x_1, x_2))=0,\quad x_1, x_2 \in I.
\end{equation}

To complete the proof, we have to construct an 
 operator $\gamma_B: B \otimes A \to B$ that satisfies
the equation
\begin{equation}
\label{for_gamma_B}
d_{\hoch}\gamma_B(b,a,a')=
\alpha_B(\alpha_B(b, a), a') - \alpha_B(b, \alpha_A(a, a')) + b \gamma_A(a, a'),
\end{equation}
where the Hochschild differential
$d_{\hoch}:\ Hom_\C(B \otimes A, B) \to Hom_\C(B \otimes A \otimes A, B)$
is given by the formula
$d_{\hoch}\gamma_B(b,a,a'):=
\gamma_B(ba, a') - \gamma_B(b, aa')+ \gamma_B(b, a) a'.$

Let  $\eta(b, a, a')$ denote the RHS of \eqref{for_gamma_B}.
A straghtforward 
computation shows that $\eta$ is a Hochschild cocycle, explicitly,
one has
$$
\eta(ba, a', a'') - \eta(b, a a', a'') + \eta(b, a, a' a'') +
\eta(b, a, a') a'' = 0.
$$
We claim further that  $\eta$ gives the zero class in
Hochschild cohomology.
To see this we use the Hochschild-Kostant-Rosenberg isomorphism.
Thus, we must restrict
$\eta$ to $B \otimes I \otimes I$ and compute
$\alt(\eta)$. But equation \eqref{pseudoLie}
says that the RHS  of formula \eqref{for_gamma_B}
is symmetric in the last two arguments.
We conclude that $\alt(\eta)=0$. Hence,
 $\eta$ is a Hochschild coboundary, and part (iii)
follows.
\end{proof}

\begin{rem} Here we sketch another proof of part $(ii)$:
Equation \eqref{hoch-B} says that
 $d\alpha_B=\alpha_A\cdot Id_B$ holds in $C^2(A, \End_\C (B))$.
Let $D^\hdot(A, \End_\C (B))$ be a subcomplex of the Hochschild complex
formed by cochains given by multidifferential operators.
The arguments from \cite{MR2062626}, pp. 16-17 can be used
to show that the cohomology groups of the complex
 $D^\hdot(A, \End_\C (B))$ are $\wedge^\hdot N_{X/Y}$.
Therefore, it follows from a version of the
Hochschild-Kostant-Rosenberg result that the imbedding $\jmath:\
D^\hdot(A, \End_\C (B))\hookrightarrow
C^\hdot(A, \End_\C (B))$ induces an isomorphism
$H^\hdot(\jmath)$, on
cohomology. 

Now, 
the Poisson bivector on $X$  gives
a class $\alpha_A\cdot Id_B\in D^2(A, \End_\C (B))$.
Since $Y$ is coisotropic,
 $\jmath(\alpha_A\cdot Id_B)\in C^2(A, \End_\C (B))$
is a coboundary. By injectivity
of $H^2(\jmath)$
the class $\alpha_A\cdot Id_B\in D^2(A, \End_\C (B))$
is itself a coboundary, i.e. $d\alpha_B=\alpha_A\cdot
Id_B$ for some $\alpha_B\in D^2(A, \End_\C (B))$. A separate easy calculation
shows that the principal symbol of $\alpha_B$ is a linear combination of 
maps $Sym^2 \Omega_{A/\C} \to B$ and $\Omega_{A/\C} \otimes \Omega_{B/\C}
\to B$, i.e.
its component $Sym^2 \Omega_{B/\C} \to B$ is actually zero. 
\end{rem}

\begin{rem}
It can be shown using the arguments of the proof above and those of 
Section 4 below that the existence
of a (not necessarily split)
deformation for $\oo_Y$ with $Y$ coisotropic, is equivalent to the 
vanishing of a certain class in $H^1(Y, N_{X/Y})$. This class
is the cup product
of the Atiyah class in $H^1(Y, \Omega^1_Y \otimes_{\oo_Y} End(N_{X/Y}))$ 
with the image of $P$ in $H^0(Y, T_Y \otimes N_{X/Y})$.  
See Theorem 7 in \cite{BGP} for more detail. 
\end{rem}

\section{An algebraic construction of BV operators}

\subsection{Complexes computing 
${\operatorname{Tor}}_\bullet^A(B, C)$ and ${\operatorname{Ext}}^\bullet_A(B, C)$.}\label{TE}

In this subsection we fix a commutative
algebra $A$ and a pair of $A$-modules $B, C$.
We have associated algebra homomorphisms
$g: A\to\End_\C B$, resp. $h: A\to\End_\C C$.
Write $T(A)$ for the tensor algebra of the vector space $A$.

Recall that the $A$-module $B$ admits a free bar resolution 
$B \otimes T(A)  \otimes A \to B$, cf. \cite{MR1269324}. Therefore 
$\Tor_\idot^A(B, C)$ and $\Ext^\hdot_A(B, C)$ can be computed as
the cohomology 
of complexes $T_\idot$ and $E^\hdot$, respectively, with terms
$$
T_i = B \otimes A^{\otimes i}  \otimes C, \quad 
E^i = \Hom_k(B \otimes A^{\otimes i}, C)
$$
The corresponding differentials, $d_T$ and $d_E$ respectively,
 are given by
\begin{align*}
d_T(b \otimes a_1 \otimes \ldots \otimes a_i \otimes c) &=
ba_1 \otimes \ldots \otimes a_i \otimes c + (-1)^i
b \otimes a_1 \otimes \ldots \otimes a_{i-1} \otimes a_i c +\\
&+ \sum_{s=1}^{i-1} (-1)^s b \otimes a_1 \otimes \ldots \otimes a_s a_{s+1}
\otimes \ldots \otimes a_i \otimes c,\\
d_E \phi(b \otimes a_1 \otimes \ldots \otimes a_{i+1}) &=
- \phi(ba_1 \otimes \ldots \otimes a_{i+1}) + (-1)^{i-1}
\phi(b \otimes a_1 \otimes \ldots \otimes a_{i}) a_{i+1} \\
&+ \sum_{s=1}^{i-1} (-1)^{s-1} 
\phi(b \otimes a_1 \otimes \ldots \otimes a_s a_{s+1}
\otimes \ldots \otimes a_{i+1})
\end{align*}

We consider  deformations of the triple
$(A,B,C)$.
Such a deformation is determined 
by an element of the  deformation complex $C^2_{A,B,C}$ given by a cocycle
 $(\alpha_A,\alpha_B,\al_C),$
see \S\ref{def_comp_sec}. Working with $T_\idot$ we always assume that
$\alpha_B$ gives a deformation of $B$ to a right module and $\alpha_C$ 
a deformation of $C$ to a left module.

The triple $(\alpha_A, \alpha_B, \alpha_C)$ induces
an operation $\delta_\alpha: T_i \to T_{i-1}$ given essentially by the 
same formula as $d_T$ where $b a_1$ is replaced by $\alpha_B(b, a_1)$, 
resp.
$a_s a_{s+1}$ is replaced by $\alpha_A(a_s, a_{s+1})$ 
and $a_i c$ by $\alpha_C(a_i, c)$.
If, in addition $\alpha'_C: C \otimes A \to A$ gives a deformation to
a right module, then 
the triple $(\alpha_A, \alpha_B, \alpha'_C)$ induces
an operation $\delta'_\alpha: E^i \to E^{i+1}$ given by a formula similar
to $d_E$;
this time $\phi(X) a_i$ is replaced by $\alpha'_C(\phi(X), a_i)$. 

The following result is proved by direct computation
\begin{lemma}
\label{operations}
Let $\delta_\alpha$ be the operator on $T_\idot$ constructed from 
a triple $(\alpha_A, \alpha_B, \alpha_C)$. 
\begin{enumerate}
\item If $(\alpha_A, \alpha_B)$ is a cocycle in $C^2_{A,B}$ and 
$(\alpha_A, \alpha_C)$ is a cocycle in $C^2_h$ then, we have 
$$
\delta_\alpha d_T + d_T \delta_\alpha = 0.
$$
\item If $(\tilde{\alpha}_A, \tilde{\alpha}_B) - (\alpha_A, \alpha_B) 
= d (\beta_A, \beta_B)$ and $(\tilde{\alpha}_A, \tilde{\alpha}_C) - (\alpha_A, \alpha_C)
= d (\beta_A, \beta_C)$ then, we have 
$$
\delta_{\tilde{\alpha}} - \delta_\alpha = d_T \delta_\beta + \delta_\beta d_T,
$$
where
$$\delta_\beta(b \otimes a_1 \otimes \ldots \otimes a_n \otimes c)=
\be_B(b)\otimes a_1 \otimes \ldots \otimes a_n \otimes c +
$$
$$
+\sum_i (-1)^{i}b \otimes a_1 \otimes \ldots \otimes\be_A(a_i)\o
\ldots \o a_n\o c+ (-1)^{n+1}
b \otimes a_1 \otimes \ldots \otimes a_n \otimes \be_C(c).$$
\item If \eqref{integrab-A}, \eqref{integrab-B} hold (with 
similar equation and notation assumed for $\alpha_C$), then, we have 
$$
\delta^2_\alpha = d_T \delta_\gamma + \delta_\gamma d_T.
$$ 
\end{enumerate}

Similar identities hold for the map $\delta'_\alpha$ constructed
from $(\alpha_A, \alpha_B, \alpha_C')$, with $d_T$ replaced by $d_E$. 
\qed
\end{lemma}

\bigskip
\noindent
We now interpret $\delta_\alpha$ in the context of the 
long exact sequence of Section 1.2. 
Since
$B_\ep$ is flat over $\kep$ we can construct a bar resolution using
tensor products over $\kep$:
$$
\ldots \to B_\ep \otimes_{\kep} A_\ep \otimes_{\kep} A_\ep 
\to B_\ep \otimes_{\kep} A_\ep \to B_\ep \to 0
$$
where the bar differential is defined using the deformed
product $A_\ep \otimes_{\kep} A_\ep \to A_\ep$ and the 
deformed action $B_\ep \otimes_{\kep} A_\ep \to B_\ep$. 
In particular, $\Tor_i^{A_\ep}(B_\ep, C_\ep)$ is the homology of the complex
with the $i$-th term
$$
T_i^{\ep} = B_\ep \otimes_{\kep} A^{\otimes_{\kep} i}_\ep \otimes_{\kep} C_\ep
\simeq \big[B\otimes A^{\otimes i} \otimes C \big] \oplus
\ep\big[B\otimes A^{\otimes i} \otimes C \big] = T_i \oplus \ep T_i
$$

It is easy to see that the differential of this complex is 
$d_\ep = d + \ep \delta_\alpha$. 
The spectral sequence of the filtered  complex
(with the two step filtration) 
$\ep T_\idot \subset T^\ep_\idot$ boils down to the long exact sequence
$$
\ldots \to H_i (T_\idot, d) \to H_i (T^\ep_\idot, d_\ep) \to 
H_i (T_\idot, d) \to H_{i-1} (T_\idot, d) \to \ldots
$$
By definition, we have 
 $H_i(T_\idot, d)= \Tor^A_i(B, C)$. The connecting differential
$\delta: H_i(T_\idot, d) \to H_{i-1}(T_\idot, d)$ is computed as usual:
we take a representative $x \in T_i \subset T_i^\ep$
and assume
that $d x = 0$, hence $d_\ep x  \in \ep T_i \in T^\ep_i$.
Then,  we 
let  $\delta (x)$  be represented by
the element $\frac{1}{\ep} d_\ep x$. 
Because of
the definition of $d_\ep$ this is precisely $\delta_\alpha$.

\bigskip
\noindent
In the case of $\Ext^\hdot_A(B, C)$, we assume that
both $B, C$ are deformed as right modules.  Again, if $(\alpha_A, \alpha_B, \alpha_C')$ 
satisfies the cocycle condition determining the first order deformation, 
the operation $\delta'_\alpha$ descends to $\delta'$ on $\Ext^\hdot_A(B, C)$
defined in Section 1.2. 
In fact, $\Ext^i_{A_\ep}(B_\ep, C_\ep)$ may be computed as the cohomology of 
the complex
\begin{multline*}
E^i_\ep = \Hom_{A_\ep} (B_\ep \otimes_{\kep} 
A_\ep^{\otimes_{\kep} i} \otimes_{\kep} A_\ep, C_\ep) 
= \Hom_{\kep}(B_\ep \otimes_{\kep} 
A_\ep^{\otimes_{\kep} i}, C_\ep)\\
= \Hom_{\kep} ([B \otimes A^{\otimes i}] \oplus \ep [B \otimes A^{\otimes i}], 
C \oplus \ep C)\\
=\Hom_k (B \otimes A^{\otimes i}, C) \oplus 
\ep \Hom_k (B \otimes A^{\otimes i}, C) =: E^i \oplus \ep E^i.
\end{multline*}
The differential $d_\ep$ again splits into $d_E + \ep \delta'_\alpha$.
Hence, the connecting differential 
$$
\Ext^{i-1}_A(B, C) \to 
\Ext^i_A(B, C)
$$ 
is induced by $\delta'_{\alpha}$.

\bigskip

Observe that for $\delta$, resp. $\delta'$, part (1) of the
Lemma \ref{operations} implies that $\delta_\alpha$, resp. $\delta'_\alpha$, 
does descend to (co)homology. Part (2) says that the operator on cohomology 
does not change if $(\alpha_A, \alpha_B, \alpha_C)$ is replaced
by $(\alpha_A, \alpha_B, \alpha_C) + d(\beta_A, \beta_B, \beta_C)$. 
Part (3) says that integrability conditions imply $\delta^2=0$, resp. 
$(\delta')^2 = 0$.

\subsection{Multiplicative properties of $T_\bullet$ and $E^\bullet$.}
We begin with a  general result.

\medskip
\noindent
Let $(D, d)$ be a differential associative
$\Z/2\Z$-graded algebra with an odd differential $d$  and 
an odd linear map
$\delta: D \to D$ such that $\delta(1)=0.$
We write $|x|\in \Z/2\Z$ for the parity of  a homogeneous element $x\in D$,
and  introduce the notation
 $[d,\de]_{_+}:=d\delta + \delta d$.
Define a
 bracket $[-,-]:\ D\times D\to D$ as follows
$$
[x, y]= \delta(xy) - \delta(x)y  - (-1)^{|x|}x \delta(y)\qquad
x,y\in D.
$$ 
Also,  for any  homogeneous elements $x,y,z\in D,$ put
$$\Xi(x,y,z):=\delta(xyz) - (-1)^{|x|}x \delta(yz) -  
\delta(xy)z - (-1)^{|y| (|x| - 1)} y \delta(xz). 
$$

A straightforward computation yields the following result
\begin{lemma}
\label{ident}
The following identities hold:
\begin{multline*}
\mbox{(1)}\qquad d[x, y] - [dx, y] - (-1)^{|x|}[x, dy] = 
[d,\de]_{_+}(xy) - 
[d,\de]_{_+}(x) y - x [d,\de]_{_+}(y)
\end{multline*}
\begin{multline*}
\mbox{(2)}\qquad [x, yz] - [x, y] z - (-1)^{|y| |x|}y [x, z]
= \Xi(x,y,z)\\
+ (-1)^{|x| + |y|}xy \delta(z) + 
(-1)^{|x|}x \delta(y) z +  \delta(x)yz.
\end{multline*}
\begin{multline*}
\mbox{(3)}\qquad[[x, y], z] + (-1)^{|x| (|y| + |z|)}[[y, z], x] + 
(-1)^{|z|(|x| + |y|)}[[z, x], y]\quad\phantom{x}\\
=\delta^2(xyz) 
- z \delta^2(xy) - x \delta^2(yz) - y \delta^2(xz) + yz \delta^2(x) + 
xz \delta^2(y) + zy \delta^2(z)\\
+ \delta(\Xi(x, y, z)) - \Xi(\delta(x), y, z) - 
(-1)^{|x|} \Xi(x, \delta(y), z) 
- (-1)^{|x| + |y|}\Xi(x, y, \delta(z)).
\end{multline*}
\end{lemma}
\bigskip

Recall next that the algebra structure on Tor-groups
may be defined using the shuffle product. In more detail,
according to \cite{MR1269324}, Exercise 8.6.5, Section 8.7.5 and
Lemma 8.7.15 (as well as a similar statement for $\Ext$ groups),
one has

\begin{lemma}\label{lod} \vi The algebra $\Tor^A_\idot(B, C)$ is 
isomorphic to the homology of the
DG algebra $T_\idot$ with the shuffle product $\bullet: T_i \otimes T_j \to 
T_{i+j}$ given by 
\begin{multline*}
(b \otimes a_1 \otimes \ldots \otimes a_i \otimes c) 
\bullet (b' \otimes a_{i+1} \otimes \ldots \otimes a_{i+j} \otimes c') = \\
\sum_{\sigma \in Sh(i, j)} (-1)^\sigma
b b' \otimes a_{\sigma(1)} \otimes \ldots \otimes a_{\sigma(i+j)} \otimes c c'.
\end{multline*}

\vii The $\Tor^A_\idot(B, C)$-module $\Ext^\hdot_A(B, C)$ is isomorphic to
the cohomology of the $T_\idot$-module $E^\hdot$ with the action 
$\bullet': T_i \otimes E^j \to E^{j-i}$ given by 
\begin{multline*}
(b \otimes a_1 \otimes \ldots \otimes a_i \otimes c) 
\bullet' \phi (b' \otimes a_{i+1} \otimes \ldots \otimes a_{j}) =\\
\sum_{\sigma \in Sh(i, j)}  (-1)^\sigma 
\phi(bb' \otimes a_{\sigma(1)} \otimes \ldots \otimes a_{\sigma(i+j)})c.
\end{multline*}
\end{lemma}
\medskip

In order to be able to apply 
Lemma \ref{ident} in the
situation we are interested in, we need the following
result.

\begin{lemma}
\label{poisson}
Let $(T=T_\idot,\, d=d_T)$ be the $DG$ algebra 
computing the $\Tor$ groups and let
$\delta = \delta_\alpha$, as
in section \ref{TE}. Then,
the equation $\Xi(x, y, z) = 0$ holds for all $x, y, z$ 
if and only if, for all $b_1, b_2 \in B$, $a \in A$, 
the cochain $\alpha_B$ satisfies
\begin{equation}
\label{2nd}
\alpha_B(b_1 b_2, a) - b_1 \alpha_B(b_2, a) - b_2 \alpha_B(b_1, a) 
+ b_1 b_2 \alpha_B(1, a) = 0
\end{equation}
and the cochain $\alpha_C$ satisfies a similar identity.

Let $\alpha'_C$ be the {\em transposed}
deformation defined by the formula
$\alpha'_C(c,a):=-\alpha_C(a,c)$.
Then,  the above conditions also insure the
following identity
\begin{equation}
\label{poisson-BVmod}
\delta_\alpha(xy) m - x \delta_\alpha(y)m - y \delta_\alpha(x) m = 
\delta'_\alpha(xym) - x \delta'_\alpha(ym) - y \delta'_\alpha(xm) + 
xy \delta'_\alpha(m).
\end{equation}
\end{lemma}
\begin{proof} Let 
$
x= b_x \otimes x_1 \ldots x_{l_x} \otimes c_x
$
and use the similar notation for $y, z$. If we plug in the formula
for $\delta$ into the definition of $\Xi(x, y, z)$ we get three kinds of terms:
those which involve $\alpha_B$, $\alpha_A$ and $\alpha_C$, respectively.
For instance, the terms in $\delta(xyz)$ coming from $\alpha_A$, will involved
tensor factors of the type   
$$
\alpha_A(x_i, x_{i+1}), \alpha_A(y_j, y_{j+1}), \alpha_A(z_s, z_{s+1}), 
$$
$$
\alpha_A(x_i, y_j), \alpha_A(y_j, x_i), \alpha_A(x_i, z_s), 
\alpha_A(z_s, x_i), \alpha_A(y_j, z_s), \alpha_A(z_s, y_j)
$$
For $x\delta(yz)$ we need to include only those terms in which $\alpha_A$ is
applied to $y_j$ ans $z_s$ but not to $x_i$, and so on. Hence the terms in 
$\Xi(x, y, z)$ which depend on $\alpha_A$ cancel out by inclusion-exclusion
formula. 
Looking at terms which involve $\alpha_B$ we get for $\delta(x, y, z)$:
$$
\alpha_B(b_x b_y b_z, x_1), \alpha_B(b_x b_y b_z, y_1),
\alpha_B(b_x b_y b_z, z_1)
$$
For $x \delta(yz)$ we get 
$$
b_x \alpha_B(b_y b_z, y_1) + b_x \alpha_B(b_y b_z, z_1)
$$
and similarly for other summands in $\Xi(x, y, z)$. Extracting the 
terms which only contain $x_1$ we get 
$$
\alpha_B(b_x b_y b_z, x_1) - b_y \alpha_B(b_x b_z, x_1) - 
b_z \alpha_N(b_x b_y, x_1) + b_y b_z \alpha_B(b_x, x_1) = 0
$$
For $b_x = 1$ this gives \eqref{2nd}. On the other hand, if \eqref{2nd}
holds then
\begin{multline*}
\alpha_B(b_x b_y b_z, x_1) = b_y \alpha_B(b_x b_z, x_1) + 
b_x b_z \alpha_B(b_y, x_1) - b_x b_y b_z \alpha_B(1, x_1)\\
= b_y \alpha_B(b_x b_z, x_1) + b_z \big[b_x \alpha_B (n_y, x_1) 
- b_x b_z \alpha_B(1, x_1) \big]\\
= b_y \alpha_B(b_x b_z, x_1) + b_z \big[\alpha_B (b_y b_x, x_1) 
- b_y \alpha_B(b_x, x_1) \big] 
\end{multline*}
as required. 

The calculation for \eqref{poisson-BVmod} is similar: for 
terms involving $\alpha_B: B \otimes A \to B$ we get precisely \eqref{2nd}. 
Comparing the terms involving $\alpha_C: A \otimes C \to C$ and 
$\alpha_C': C\otimes A \to C$ we get the condition \eqref{2nd} for
$\alpha_C'$ plus the equation
$$
c_3 \alpha_C(a, c_1 c_2) - c_3 c_2 \alpha_C(a, c_1) = 
- \alpha_C'(c_1 c_2 c_3, a) + c_2 \alpha_C'(c_1 c_3, a) 
$$
Since $\alpha_C'$ is transposed to $\alpha_C$ this equation 
can also be reduced to  
\eqref{2nd} for $\alpha_C'$.  \end{proof}

\section{Proofs of main results}
\subsection{The transposed deformation}
\label{transposed} Fix a flat $\kep$-algebra deformation
$\aa$, of $\oo_X$.
Associated with  any deformation
$\cc$, of the sheaf $\oo_Z$ to a left  $\aa$-module, 
there is a transposed
deformation $\cc^t$, which gives a sheaf of right $\mathcal{A}$-modules.

To explain the definition of  $\cc^t$, recall first that
any deformation  $\cc$
 admits local $\k$-linear 
splittings (in the Zariski topology) 
$\cc \simeq \oo_Z \oplus \ep \oo_Z$. 
So, locally, the deformed module structure can be written
in the form
$(a \oplus  0)* (c \oplus 0)  = ac \oplus \ep \alpha(a, c)$.
Furthermore, we will see in Proposition \ref{exist} below
that the cochain
$\alpha(a, c)$ can be chosen to be
 an algebraic differential operator in each of its 
arguments (which satisfies an associativity condition recalled in 
Section 2).  
Thus,
 $X$ has a covering by affine open subsets $U_i$ and on each of them 
there is a splitting
 as above. It follows that, on each double intersection $U_i \cap U_j$,
 the
corresponding splittings differ by an automorphism
$$
c_1 \oplus \ep c_2 \mapsto c_1 \oplus \ep (\psi_{ij}(c_1) + c_2),
$$
where   $\psi_{ij}: \oo_Y \to \oo_Y$ is
an algebraic differential operator.
The gluing condition for the locally defined 
cochains $\alpha_i(a, c)$ and $\alpha_j(a, c)$ reads:
\beq{psi}
\alpha_i(a, c) - \alpha_j(a, c) = a \psi_{ij}(c) - \psi_{ij}(ac).
\eeq

Conversely, given a covering of $X$,  a collection of
$\alpha_i$'s describing a deformation of $\oo_Y|_{U_i}$ to a left 
module over $\mathcal{A}|_{U_i}$, and a set of operators $\psi_{ij}$
such that \eqref{psi} holds and, moreover, for each
 triple $(i,j,k)$, one has
$\psi_{ij} + \psi_{jk} = \psi_{ik}$, gives rise to
a deformation $\cc$ of $\oo_Z$ to a left $\mathcal{A}$-module 
$\cc$ if the above gluing conditions are satisfied. 
A similar statement holds for right  $\mathcal{A}$-module 
deformations as well.

Now, given a left  $\mathcal{A}$-module deformation
$\cc$, we define  $\cc^t$, the transposed right  $\mathcal{A}$-module 
deformation, by gluing locall deformations
given, on each $U_i$, 
by the formula 
$$
(c \oplus 0) * (a \oplus 0) = ac \oplus - \ep \alpha(a, c),
\quad\text{i.e. we put $\alpha^t_i(c, a) := - \alpha_i(a, c)$}.
$$
Here, the minus signs appear since the opposite algebra $\mathcal{A}^{op}$ may 
be viewed as a deformation coming from the bivector $-P$. 
The above defined local deformations are related, on double
intersections, via the operators $\psi_{ij}^t := - \psi_{ij}$.

\subsection{Proof of Theorem \ref{BV}(i): BV differential}
Given $(\mathcal A,\mathcal B,\mathcal C)$, a triple of deformations 
as in the theorem,
we use the construction of the map $\delta:
{\scr T}\!or^{\oo_X}_\idot(\oo_Y,\oo_Z)$ $\to
{\scr T}\!or^{\oo_X}_{\idot-1}(\oo_Y,\oo_Z)$ from section
\ref{bvdif}. The construction being local,
the map $\delta$ restricts to a similar
map on any open affine subset of $X$.
Clearly, in order to verify the required properties
of $\delta$, it is sufficient to verify them
for an open affine covering of $X$.

This way, we reduce the proof to the case where $X = \Spec A,$
resp.
$Y=\Spec B$ and $Z=\Spec C.$
For affine varieties all the deformations involved are 
automatically {\em split} over $\k$.
It follows that these deformations may be written in terms of
 certain cocycles $\alpha_A, \alpha_B$, and $\alpha_C$,
respectively, as in Section \ref{setup}.
Now, we are in the setting of Sections 3.2-3.3.
In particular, we may use the DG algebra $(T_\idot,d_T)$
and Lemma \ref{lod} for computing the algebra
$Tor^A_\idot(B,C)$
and we may interpret the map
$\delta$ in terms of the Bar construction.

Observe next that, thanks to
Lemma \ref{operations}(2), replacing deformations
by equivalent deformations doesn't affect the
conclusion of the theorem.
Therefore, we may adjust our cocycles using
Proposition \ref{exist} to insure that:
(1) the cocycle  $\alpha_A$ be equal to the Poisson bracket,
in particular, it can be extended to a second order deformation
$(\alpha_A,\gamma_A)$;
and (2) the cocycle $ \alpha_B$, resp. $\alpha_C$,
can also be extended to a second order deformation
and, moreover, it is given by a bidifferential operator,
as in Proposition \ref{exist}  $(ii)$. Here, the existence
of extensions to second order deformations means that
there exist $\gamma_B$ and $\gamma_C$ such that
for $\alpha:=\alpha_A\oplus\alpha_B\oplus\alpha_C$,
in the deformation complex $C^\hdot_{A,B,C}$, see \eqref{c_gh}, we have
$[\alpha,\alpha]=d_{A,B,C}(\gamma_A\oplus\gamma_B\oplus\gamma_C)$;
in particular, the second order deformations 
of $B$ and $C$ given by $(\alpha_B, \gamma_B)$
and $(\alpha_C, \gamma_C)$, respectively, correspond to the
\textit{same} second order deformation of $A$ 
given by a pair $(\alpha_A, \gamma_A)$, as guaranteed by 
Proposition \ref{exist} $(iii)$.

Recall that the existence of a second order extension 
of the deformation given by $ (\alpha_A,\alpha_B)$
is equivalent to  equations \eqref{integrab-A}-\eqref{integrab-B}.
Hence, we deduce $[d_T,\delta]_+=0$,
by Lemma \ref{operations}(1). Also,
equation $\delta^2=0$ follows
from Lemma \ref{operations}(3).
Further, the constraints from Proposition \ref{exist}$(ii)$
on the
order of  bidifferential operators
insure that
the
assumption $\Xi(x,y,z)=0$, of Lemma \ref{poisson},
holds in our case. Hence,
combining Lemma \ref{poisson} with
Lemma \ref{ident}, we deduce 
that   the Poisson and  Jacobi identities hold already
in the algebra $T_\idot$.
We conclude that these  identities hold
in the $\Tor$ algebra as well, and the theorem follows. 
\qed
\medskip

The proof of the properties of the differential
$\delta'$ on the Ext-sheaves
 is completely similar and is left for the reader.

\subsection{Proof of Theorem \ref{BV}(ii)-(iii): Gerstenhaber structures}
Let $Y,Z$
be smooth coisotropic submanifolds of a smooth
Poisson variety $X$. We must prove the following

\begin{thm}\label{7} 
The sheaf of graded algebras ${\scr{T}\!}or^{\mathcal{O}_X}_\idot
 (\mathcal{O}_Y, 
\mathcal{O}_Z)$ admits a canonical structure of a Gerstenhaber algebra, i.e.
a graded symmetric bracket of degree (-1) which satisfies Poisson and Jacobi
identities.

Similarly, the sheaf $\scr{E}^{\!}xt^\hdot_{\mathcal{O}_X}
(\mathcal{O}_Y, \mathcal{O}_Z)$ has a canonical structure of a Gerstenhaber
module, i.e. a bracket 
$[\cdot, \cdot]': {\scr{T}\!}or_i \otimes \scr{E}^{\!}xt^j 
\to \scr{E}^{\!}xt^{j-i+1}$ such that
$$
[xy, m]' = x [y, m]' + (-1)^{\deg y \deg x} y [x, m]'; 
$$
$$
[x, ym]' = [x, y] m + (-1)^{\deg y \deg x} y [x, m]'
$$
$$
[[x,y], m]' = [x, [y, m]']' + (-1)^{\deg y \deg x}[y, [x, m]']'
$$
\end{thm}
\begin{proof} Let $\mathcal{A}$ be the standard (split)
algebra deformation of the structure sheaf $\oo_X$ given by
the formula $f*g=fg+\frac{\ep}{2}\{f,g\}$.

On $X$, we choose
an  affine open covering $\{U_i\}$ such that each open subset $U_i$ has
trivial canonical class.
By Section 2, on each $U_i$,
we can find deformations $\mathcal{B}_i, \mathcal{C}_i$ which extend
to second order deformations.
Writing $ A_i$, resp. $ B_i, C_i$, for the
 corresponding algebras of global sections and
applying Theorem \ref{BV}, we get a BV algebra
structure on 
$\Tor^{ A_i}_\idot( B_i,  
C_i)$. By   Lemma \ref{operations}(2),
the corresponding
BV differential $\delta_i$ on the $\Tor$ algebra 
is unaffected by a  change of cocycles $\alpha_A, \alpha_B,
\alpha_C$ provided neither 
 the cohomology class of $\alpha_A \oplus \alpha_B$
nor  the cohomology class of $\alpha_A \oplus \alpha_C$
is changed. 
In particular, we may insure that all the
cocycles involved satisfy the conclusions of Proposition \ref{exist}.

Next, fix a double intersection 
$U_i \cap U_j$. We must show that the
restrictions to $U_i \cap U_j$ of the two BV differentials
$\delta_i$ and $\delta_j$,  arising from
the triples $(\alpha_A)_i,\,(\alpha_B)_i,\,(\alpha_C)_i$
and 
$(\alpha_A)_j,\,(\alpha_B)_j,\,(\alpha_C)_j$
respectively,
induce the same Gerstenhaber bracket on the sheaf
${\scr{T}\!}or^{\oo_X}_\idot(\oo_Y,\oo_Z)|_{U_i \cap U_j}.$

To this end, we observe that the deformation
$\mathcal A$ being split and globally defined, on
$U_i \cap U_j$, we have $(\alpha_A)_i=(\alpha_A)_j$.
Next, consider the complex $C^\hdot_{A,B}$,
cf. \eqref{d_g}, associated with
the algebras $A_{ij}=\k[U_i \cap U_j]$
and $B_{ij}=\k[Y\cap U_i \cap U_j]$.
The injectivity claim
at the beginning of Proposition
\ref{exist} implies that
the cocycles $(\alpha_A)_i\oplus(\alpha_C)_i$
and 
$(\alpha_A)_j\oplus(\alpha_C)_j$ give equal
cohomology classes in $H^2(C^\hdot_{A,B})$. We conclude that there exist
cochains
$((\beta_A)_{ij}, (\beta_B)_{ij})$ such that, in $C^2_{A,B}$, we have 
\beq{C}
\big(0 \oplus (\alpha_C)_i - (\alpha_C)_j\big) = d_{A,B} ((\beta_A)_{ij}\oplus
(\beta_B)_{ij}).
\eeq
We see that adjusting the restriction of the
triple $(\alpha_A)_i, (\alpha_B)_i,
(\alpha_C)_i$ to $U_i \cap U_j$ by 
a coboundary allows to achieve that $(\alpha_C)_i = (\alpha_C)_j$.

\begin{rem} At this point, the reader should be alerted that,
although an equation similar to \eqref{C}
holds for the difference
$0 \oplus (\alpha_B)_i - (\alpha_B)_j$ as well,
that  equation may require 
a completely different choice of the
cocycle $(\beta_A)_{ij}$. Thus, on $U_i\cap U_j$,
the cocycles in the complex $C^\hdot_{A,B,C}$,
cf. \eqref{c_gh},
corresponding to
$(\alpha_A)_i\oplus(\alpha_B)_i\oplus(\alpha_C)_i$
and 
$(\alpha_A)_j\oplus(\alpha_B)_j\oplus(\alpha_C)_j$
respectively, do not necessarily give equal
cohomology classes,
in general.
\end{rem}

The arguments above show that we may assume,
adjusting by coboundaries if necessary,
that we have $(\alpha_A)_i=(\alpha_A)_j$
and $(\alpha_C)_i=(\alpha_C)_j$
(but not also $(\alpha_B)_i=(\alpha_B)_j$ at the same time).
Further, let   $C^\hdot_h$
be the deformation complex for the pair $(A,B)$,
where $h: A\to \End B$ is
the obvious homomorphism.
Taking the difference
of the cocycles $(\alpha_A)_i\oplus(\alpha_B)_i$
and $(\alpha_A)_j\oplus(\alpha_B)_j$
yields a cocycle in  $C^2_h$ of the form
$0\oplus (\alpha_B)_i - (\alpha_B)_j$.
Since the first component here
is $0$, the cocycle condition  says, cf. \eqref{d_g},
that $(\alpha_B)_i - (\alpha_B)_j:\ B\o A\to B$
is a Hochschild 1-cocycle.
But any Hochschild 1-cocycle has the form $(b \otimes a) \mapsto 
b\cdot \xi(a)$ for a certain derivation $\xi: A_{ij} \to A_{ij}$. Thus, 
the difference of two BV differentials $\delta_i - \delta_j$ on 
$U_i \cap U_j$ is induced by the operator
$$
\tilde{\delta}(b \otimes a_1 \otimes \ldots \otimes a_i \otimes c) = 
b\cdot  \xi(a_1) \otimes a_2 \otimes \ldots \otimes a_i \otimes c 
$$

It is straightforward to verify that, since $\xi$ is
a first order differential operator,
any operator $\tilde{\delta}$ as above
induces the zero bracket on $\Tor$. Hence the 
brackets induced by $\delta_i$ and $\delta_j$ agree on $U_i\cap U_j$. 
This means that the Gersenhaber bracket is independent on the local choice 
of $\alpha_B$ and $\alpha_C$, which  finishes the proof
of the first part of the theorem.

We turn to the second part  of the theorem concerning
Gerstenhaber modules.
Observe that the module bracket can be defined via
$$
[x, m]' = \delta'(xm) - \delta(x) m - (-1)^{\deg x}x \delta'(m).
$$

Since arbitrary $\C$-linear maps $B \otimes A \to B$, etc., 
do not localize in general, we need to work with the subcomplex of $E^\hdot$
given by polydifferential operators (by the last remark of Section 2 this
subcomplex has the same cohomology). Observe that our deformations of $Y$
and $Z$ are indeed given locally by bidifferential operators.

Now, both versions of Poisson identity are equivalent to
\eqref{poisson-BVmod}.
The Jacobi identity on the module follows from $(\delta')^2 = 0,$ 
once the Poisson identity is established. The remaning part of the proof is
the same as for $\Tor$. 
\end{proof}

\section{Examples.}\label{Exam}

The proof of the previous Theorem shows that sometimes the BV structure 
on $\Tor$ or $\Ext$ is well-defined globally. 
This is the case, for instance, whenever
in the setting of the proof of Theorem
\ref{7},
the cocycles  $\alpha_C$ \textit{and} 
$\alpha_B$,  are defined globally.

Here is one such example.

\subsection{Koszul bracket} 
Take $X = Y \times Y$ and view $Y\sset X$ as the diagonal. 
Then, one easily finds
$$
{\scr{T}\!}or^{\mathcal{O}_X}_\idot(\mathcal{O}_Y, 
\mathcal{O}_Y)\simeq  \Omega^\hdot_{Y}; \qquad
{\scr{E}\!}xt^\hdot_{\mathcal{O}_X} 
(\mathcal{O}_Y, \mathcal{O}_Y)  \simeq \La^\hdot 
\Omega^\vee_{Y}.
$$

\begin{prop}
For any Poisson bivector field $P \in H^0(Y, \La^2T_Y)$,
the induced BV differential $\delta$ on $\Omega^\hdot_{B/k}$ is given by 
the formula $\delta = i_P\, d_{DR} + d_{DR}\, i_P$,
where the
de Rham differential $d_{DR}$ 
has degree $+1$, and where $i_P$
is a degree $(-2)$ contraction operation (with 
the bivector $P$). The BV differential $\delta'$ on 
$\La^\hdot \Omega^\vee_{B/k}$ is given by the Schouten bracket
$[P, -]$.
\end{prop} 
\noindent
\begin{proof} To simplify notation we will work in the affine case
although all formulas make sense globally. Thus we consider
$A = B \otimes B$ with the quotient map $m: B \otimes B \to B$
given by the product. Also, take $C = B$. By standard 
results, e.g. \cite{MR1269324}, we have a Hochschild cocycle on $A$:
$$
\alpha_A(x \otimes x', y \otimes y') = \frac{1}{2}P(dx \wedge dy) 
\otimes x' y ' - xy \otimes \frac{1}{2}P(dx' \wedge dy')
$$
and the diagonal $Y_\Delta  \subset X = Y\times Y$ is coisotropic 
with respect to the corresponding Poisson structure. We also have a 
right deformation of $B$ induced by
$$
\alpha_B(b, x \otimes y) = - \frac{1}{2}P (d(bx) \wedge dy) + 
\frac{1}{2}P (db \wedge dx) y:
B\otimes A \to B
$$
For the second argument of $\Tor^A_\idot(\cdot, \cdot)$ we use the
transposed map $\alpha_B^t: A\otimes B \to B$. Our goal is to compute
the induced BV differential on $\Omega^\hdot_{B/k}$. To that end, we
need explicit quasi-isomorphisms between $\Omega^\hdot_{B/k}$
and $T_\idot$.

Observe that usually $\Omega^\hdot_{B/k}$ is identified with 
the cohomology of $C_\idot(A, A) = A^{\otimes (\hdot + 1)}$ and the 
standard Hochschild differential. Our complex $T_\idot$ is slightly 
different although quasi-isomorphic to $C_\idot(A, A)$ by the map:
$$
b \otimes a_1 \otimes \ldots \otimes a_n \otimes b' 
\mapsto (b' b) \otimes m(a_1) \otimes \ldots \otimes m(a_n)
$$ 
Combining we the Hochschild-Kostant-Rosenberg
 isomorphism we have a pair of mutually inverse 
quasi-isomorphisms:
$$
b \otimes a_1 \otimes \ldots \otimes a_n \otimes b' \mapsto 
\frac{1}{n!} b'b \cdot d m(a_1) \wedge \ldots \wedge dm (a_n): 
T_\idot \to \Omega^\hdot_{B/k}
$$
and the map 
$$
b_0 d b_1 \mapsto b_0 \otimes (b_1 \otimes 1)  \otimes 1: \Omega^1_{B/k}
\to T_1
$$
extended multiplicatively. For example, $b_0 db_1 \wedge db_2 \wedge db_3$
maps to the antisymmetrization of 
$$
b_0 \otimes (b_1 \otimes 1) \otimes (b_2 \otimes 1) \otimes (b_3 \otimes 1) 
\otimes 1 \in T_3
$$
in the three middle terms. The assertion follows from the above definitions
of $\alpha_A$ and $\alpha_B$ and a straightforward computation.

\medskip
\noindent
The case of $\delta'$ is entirely similar. 
\end{proof}

\bigskip
\noindent
We observe here that the differential $\delta$ of the above proposition 
was first constructed from a Poisson bivector $P$ by Koszul in \cite{MR837203}.
Also, the differential $\delta$ (and not just the induced Gerstenhaber bracket)
is canonically defined since the two arguments of $\Tor^A_\idot(B, B)$
are taken with their conjugate deformations. 
We also remark that the (co)homology of the differentials $\delta'$ and
$\delta$ in this case are called Poisson cohomology and homology, 
respectively.

\subsection{Self-intersection of a coisotropic submanifold} 
Let $X$ be an arbitrary Poisson variety, and let $Y=Z$,
a coisotropic subvariety.
In this case, we have
$\scr{E}^{\!}xt^\hdot_{O_X} (\mathcal{O}_Y, \mathcal{O}_Y) = 
\La^\hdot (N_{X/Y})$. 

The proof of Theorem \ref{7} shows that 
the differential $\delta'$ is well-defined globally: although 
$\alpha_B(b, a)$ exist in general only locally, on double intersections
of an affine cover the difference between two choice of $\alpha_B$
is a coboundary (since $B= C$, where $B$ is the
affine counterpart of $\oo_Y$ and
$C$  is the
affine counterpart of $\oo_Z$, as usual).
Hence these two choices give
the same $\delta'$ on cohomology. 
The subsheaf  $\mathcal{K} \subset \La^\hdot(T_X)$ 
formed by vector fields which project to zero in $\La^\hdot(N_{X/Y})$,
is a Lie subalgebra with respect to the Schouten bracket. 
Then $\mathcal{K}$  
acts on $\La^\hdot(N_{X/Y})$  since any Lie subalgebra acts on 
a quotient by itself. 

\begin{prop}
The BV differential on $\La^\hdot(N_{X/Y})$
is given by the action of the Poisson bivector.
\end{prop}
\noindent
\begin{proof} It suffices to prove the statement on an affine open
subset. Observe that $\Ext^\hdot_A(B, B)$
is computed by the subcomplex 
$$
C^{\hdot -1}(A, \End_\C B) \subset 
C^\hdot_{A,B}.$$ 
The differential $\delta'$ 
of $\gamma_B: A^{\otimes (n -1)} \to \End_\C B$ is explicitly given 
by 
$$
\delta'(\gamma_B) = [\alpha_A \oplus \alpha_B, 0 \oplus \gamma_B]:
A^{\otimes n} \to \End_\C B
$$
where $[\cdot, \cdot]$ is the Lie bracket on $C^n_{A,B}$ 
introduced in Section 2.2. Now a straightforward calculation finishes
the proof. 
\end{proof}

\subsection{Symplectic 2-form need not be locally exact}
\label{sasha}
The following construction of a large class of
examples of algebraic closed differential forms
which are not  locally exact in \`etale topology
was explained to us by A. Beilinson.

Let $X$ be a smooth algebraic
variety, and $\omega$ a non-zero
holomorphic $i$-form on $X$.
 We assume, in addition,  that 
for some compactification of $X$
the form $\omega$ has logarithmic singularities at infinity (it is
then automatically closed
thanks to Deligne results \cite{De}). In any case, one can merely assume that $X$ is
itself a projective variety.

\begin{claim} For any \`etale morphism $\pi: U\to X$, the $i$-form 
$\omega_U=\pi^*\omega$, on $U$, is \textit{not} exact.
\end{claim}

\begin{proof} Let $X\into Z$ be a compactification 
such that $Z\setminus X$ is a divisor with normal crossings. By
Hironaka's resolution of singularities, one
 can find a similar compactification $U\into T$ 
such that the map $\pi: U\to X$ extends to a morphism $T\to Z$. 
It is known that in such a case,
the $i$-form $\omega_U$ has logarithmic
 singularities at $T\smallsetminus U$.
Furthermore, this form is clearly
non-zero since its restriction to $U$ is non-zero.
The map sending a differential
form $\beta$ on $T$ with logarithmic singularities at $T\smallsetminus U$
to $[\beta]\in H^i (U)$,
the corresponding de Rham
 cohomology class,
 is known to be {\em injective}, by [De].
So, the class  $[\omega_U] \in H^i (U)$ is non-zero.
We conclude that $\omega_U$ can not be an exact form,
as claimed.
\end{proof}

The above result produces examples of smooth symplectic varieties
(e.g. $X$  an abelian surface) such that the corresponding
 symplectic 2-form cannot be made exact by passing to any
 \`etale
open subset.

\subsection*{Acknowledgments.}
{\small{We are grateful to  Vasiliy Dolgushev for his
suggestion to use the shuffle product and
to Sasha Beilinson for the construction
reproduced in \S\ref{sasha}. We are
indebted to Lev Rozansky for many patient explanations
of his ongoing (partially unpublished) work with Anton Kapustin.

The work of the first author was supported in part by the
Sloan Research Fellowship and the work of the second author
 was  supported in part by the NSF   grant  DMS-0601050.}}

\noindent
\textit{Addresses: \\
Department of Mathematics, University of  California at Irvine, Irvine, CA 92697, USA;\\ Department of Mathematics, University of Chicago,  Chicago, IL 60637, USA.}

\bigskip
\noindent
\textit{Email: vbaranov@math.uci.edu; ginzburg@math.uchicago.edu}


\begin{thebibliography}{10}

\bibitem{BGP}
V.~Baranovsky, V.~Ginzburg, and J.~Pecharich, \emph{Deformation of line bundles
  on coisotropic subvarieties} (2009). Preprint arXiv:0909.5520.

\bibitem{BF}
K.~Behrend and B.~Fantechi, \emph{Gerstenhaber and Batalin-Vilkovisky
  structures on Lagrangian intersections}, in Algebra, Arithmetic and Geometry:
  Manin Festschrift, Volume I., Vol. 269 of \emph{Progress in Mathematics},
  Birkh\"auser (2009).

\bibitem{De}
P.~Deligne, \emph{Th\'eorie de {H}odge. {II}}, Inst. Hautes \'Etudes Sci. Publ.
  Math.  (1971), no.~40,  5--57.

\bibitem{FMY}
Y.~Fr\'egier, M.~Markl, and D.~Yau, \emph{The $L_\infty$-deformation complex of
  diagrams of algebras}, New York Journal of Math  (2009), no.~15,  353--392.

\bibitem{MR704600}
M.~Gerstenhaber and S.~D. Schack, \emph{On the deformation of algebra morphisms
  and diagrams}, Trans. Amer. Math. Soc. \textbf{279} (1983), no.~1,  1--50.

\bibitem{MR0238860}
A.~Grothendieck, \emph{\'{E}l\'ements de g\'eom\'etrie alg\'ebrique. {IV}.
  \'{E}tude locale des sch\'emas et des morphismes de sch\'emas {IV}}, Inst.
  Hautes \'Etudes Sci. Publ. Math.  (1967), no.~32,  361.

\bibitem{KR}
A.~Kapustin and L.~Rozansky, \emph{Three-dimensional topological field theory
  and symplectic algebraic geometry II} (2009). Preprint arXiv:0909.3643.

\bibitem{MR2062626}
M.~Kontsevich, \emph{Deformation quantization of {P}oisson manifolds}, Lett.
  Math. Phys. \textbf{66} (2003), no.~3,  157--216.

\bibitem{MR837203}
J.-L. Koszul, \emph{Crochet de {S}chouten-{N}ijenhuis et cohomologie},
  Ast\'erisque  (1985), no. Numero Hors Serie,  257--271. The mathematical
  heritage of {\'E}lie Cartan (Lyon, 1984).

\bibitem{Lo}
J.-L. Loday, Cyclic homology, Vol. 301 of \emph{Grundlehren der Mathematischen
  Wissenschaften [Fundamental Principles of Mathematical Sciences]},
  Springer-Verlag, Berlin, second edition (1998), ISBN 3-540-63074-0. Appendix
  E by Mar{\'{\i}}a O. Ronco, Chapter 13 by the author in collaboration with
  Teimuraz Pirashvili.

\bibitem{O}
D.~O. Orlov, \emph{Triangulated categories of singularities and {D}-branes in
  {L}andau-{G}inzburg models}, Tr. Mat. Inst. Steklova \textbf{246} (2004), no.
  Algebr. Geom. Metody, Svyazi i Prilozh.,  240--262.

\bibitem{MR1269324}
C.~A. Weibel, An introduction to homological algebra, Vol.~38 of
  \emph{Cambridge Studies in Advanced Mathematics}, Cambridge University Press,
  Cambridge (1994), ISBN 0-521-43500-5.

\end{thebibliography}
\end{document}